\documentclass{amsart}
\usepackage{amssymb, verbatim}

\newtheorem{theorem}{Theorem}

\newtheorem{prop}[theorem]{Proposition}
\newtheorem{cor}[theorem]{Corollary}

\newtheorem{defn}[theorem]{Definition}

\theoremstyle{definition}
\newtheorem*{remark}{Remark}
\renewcommand{\[}{\begin{equation}}
\renewcommand{\]}{\end{equation}}
\newcommand{\ddb}{\sqrt{-1}\partial\overline{\partial}}

\def\Xint#1{\mathchoice
{\XXint\displaystyle\textstyle{#1}}
{\XXint\textstyle\scriptstyle{#1}}
{\XXint\scriptstyle\scriptscriptstyle{#1}}
{\XXint\scriptscriptstyle\scriptscriptstyle{#1}}
\!\int}
\def\XXint#1#2#3{{\setbox0=\hbox{$#1{#2#3}{\int}$ }
\vcenter{\hbox{$#2#3$ }}\kern-.6\wd0}}

\def\dashint{\Xint-}

\title{The partial $C^0$-estimate along the continuity method}
\author{G\'abor Sz\'ekelyhidi}
\address{Department of Mathematics, University of Notre Dame, Notre
  Dame, IN 46556}
\email{gszekely@nd.edu}

\begin{document}

\begin{abstract}
  We prove that the partial $C^0$-estimate holds for metrics along
  Aubin's continuity method for finding K\"ahler-Einstein
  metrics, confirming a special case of a conjecture due to Tian.
  We use the method developed in recent work of
  Chen-Donaldson-Sun on the analogous problem for conical
  K\"ahler-Einstein metrics. 
\end{abstract}

\maketitle

\section{Introduction}
A fundamental problem in K\"ahler geometry is the existence of
K\"ahler-Einstein metrics on Fano manifolds. Yau~\cite{Yau93}
conjectured that the existence is related to the stability of the
manifold in an algebro-geometric sense. A precise notion of stability,
called K-stability, was defined by Tian~\cite{Tian97}, who also showed that
K\"ahler-Einstein manifolds are K-stable. The other direction of the
conjecture was recently obtained by 
Chen-Donaldson-Sun~\cite{CDS12},
showing that K-stable manifolds admit K\"ahler-Einstein
metrics. The K\"ahler-Einstein metrics are constructed using a
continuity method suggested by Donaldson~\cite{Don09},
passing through singular K\"ahler-Einstein metrics
with conical singularities along a divisor. One key ingredient of the
proof is establishing the partial $C^0$-estimate conjectured
by Tian~\cite{Tian90}, for such conical metrics. In this paper we extend
the techniques of Chen-Donaldson-Sun~\cite{CDS13_2, CDS13_3}
to obtain the partial $C^0$-estimate along the more classical
continuity method studied by Aubin~\cite{Aub84}, which uses smooth metrics. 

First we recall the Bergman kernel. Let $M$ be a Fano manifold,
fix a metric $\omega\in c_1(M)$, 
and choose a metric $h_k $ on $K_M^{-1}$ whose curvature is
$k\omega$. Define the $L^2$ inner product on $H^0(K_M^{-k})$  given  by
\[ \langle s, t\rangle = \int_{M} \langle s,t\rangle_{h^k}
\frac{(k\omega)^n}{n!}. \]
The Bergman kernel $\rho_{\omega, k} : M\to\mathbf{R}$ can be defined as
\[ \rho_{\omega,k}(x) = \sum_{i=0}^{N_k} |s_i(x)|^2_{h^k}, \]
where $\{s_0,\ldots, s_{N_k}\}$ is any orthonormal basis of $H^0(K_M^{-k})$. 

Now fix another K\"ahler form $\alpha\in c_1(M)$ and
suppose that $\omega_t\in c_1(M)$ solves the equation
\[ \label{eq:Aubincont}
\mathrm{Ric}(\omega_t) = t\omega_t + (1-t)\alpha, \]
for $t\in [0, T)$, with $T\leqslant 1$. Our main result is that
Tian's partial $C^0$-estimate 
holds for this family of metrics. 
\begin{theorem}\label{thm:main}
  There is an integer $k_0$ and a constant $c > 0$ (depending on $M,
  \alpha$), 
  such that $\rho_{\omega_t, k_0}(x) > c$ for all $x\in M$ and $t\in [0,T)$.
\end{theorem}

This is a special case of a conjecture due to Tian~\cite{Tian90_2}, who
conjectured more generally that the same estimate holds
independently of the background K\"ahler form $\alpha$.
As has been explained by Tian in many places, e.g.~\cite{Tian97,
  Tian02}, this partial $C^0$-estimate along the continuity method is
a crucial ingredient in relating the existence of K\"ahler-Einstein metrics
to an algebro-geometric stability notion. 

The continuity method through conical K\"ahler-Einstein metrics
alluded to above can be
thought of as an analog of \eqref{eq:Aubincont}
where $\alpha$ is a current of integration
along a divisor. The advantage of this approach is that on
``most'' of the manifold the metrics are Einstein, and so powerful
convergence results due to Anderson~\cite{An90},
Cheeger-Colding~\cite{CC97, CC00, CC00_1},
Cheeger-Colding-Tian~\cite{CCT02} can be applied, at the price of
having a singularity along the divisor.

It seems likely 
that Theorem~\ref{thm:main} can be used to give an alternative proof
of the existence of K\"ahler-Einstein metrics on K-stable manifolds,
without the use of conical metrics, by adapting more arguments
from~\cite{CDS13_3}. We hope to address this in future work.
If instead of
K-stability we use the stability notion introduced by
Paul~\cite{Paul12}, then using arguments due to Tian (see
e.g. \cite{Tian13_1, Tian97, Tian12_1}), we obtain the following whose
proof will be given in Section~\ref{sec:corproof}. 
\begin{cor} \label{cor:Paul}
  Suppose that $M$ is a Fano manifold with no holomorphic vector
  fields. Then $M$ admits a K\"ahler-Einstein metric if and only if
  $M$ is stable in the sense
  defined by Paul~\cite{Paul12} for all projective
  embeddings using powers of $K_M^{-1}$. 
\end{cor}

Note that Paul's
stability can also be tested using test-configurations, just like
K-stability (see \cite{Paul12}), however in general the ``weight''
associated to a test-configuration is different in the two theories. 
In particular the weight in Paul's stability notion recovers precisely
the asymptotics of the Mabuchi energy along a one-parameter family in
a fixed projective space, in contrast with
the generalized Futaki invariant of Tian~\cite{Tian97} and
Donaldson~\cite{Don02} in K-stability. Indeed, in general the latter is only
related to Mabuchi energy asymptotics in a fixed projective space up
to a correction term (see Paul-Tian~\cite{PT06} and
Phong-Ross-Sturm~\cite{PRS06}). 

We will now sketch the general idea in obtaining a lower bound on the
Bergman kernel for a fixed metric $\omega$. 
  The Bergman kernel can alternatively be written as
\[ \rho_{\omega, k}(x) = \sup_{s\in H^0(K_M^{-k})} \frac{|s(x)|_h^2}{\Vert
  s\Vert_{L^2}^2}, \]
so to prove a lower bound on $\rho$ we need to construct
``peaked sections'' for any point $x$, which are sections with
small $L^2$-norm, but large value at $x$. 
The basic idea for this goes back to Tian~\cite{Tian90_1}, who used
H\"ormander's $L^2$-technique to construct peaked sections (see also
Siu-Yau~\cite{SY82} for a precursor), and in the
present context proved not only that $\rho_{\omega,k}$ is bounded
below for large enough $k$, but also gave precise asymptotics as
$k\to\infty$. These were later refined by Ruan~\cite{Ru98}, 
Zelditch~\cite{Zel98}, Lu~\cite{Lu98} and many others.

The idea is that once $k$ is
large, the metric $k\omega$ is well approximated by the Euclidean
metric on the unit ball $B$ near any given point $p$ (in practice a larger
ball is used, with radius small compared to $\sqrt{k}$). The metric
$h_k$ in turn is well approximated by the metric $e^{-|z|^2}$ on the
trivial bundle. The peaked section at $p$ is then constructed by
gluing the constant $\mathbf{1}$ section onto $M$ using a cutoff function, and
then perturbing it to a holomorphic section. The perturbation works
because the exponential decay of the section $\mathbf{1}$ ensures that
after the gluing our section is still approximately holomorphic in an
$L^2$-sense. Thus the key property of a K\"ahler manifold
$(M,\omega)$ is that there is some scale, depending on $\omega$,
at which a neighborhood of each point is well approximated by the Euclidean ball.

If we now have a family of Fano manifolds $(M_i, \omega_i)$, and
we want to find a $k_0$ such that $\rho_{\omega_i, k_0} > c$ for a
fixed positive $c$, then we need to find a scale at which
neighborhoods of each point in each $(M_i, \omega_i)$ are well
approximated by a model. However in general, unless we have very good
control of the metrics, the model cannot just be
the Euclidean ball anymore. Under certain conditions, when the limit
of the $(M_i,\omega_i)$ is an orbifold, the local model can be taken
to be a quotient of the Euclidean ball, and the method can still be
applied (see Tian~\cite{Tian90, Tian93}). For more general sequences
of metrics the fundamental tool is Gromov
compactness~\cite{Gro07}, and the structure theorems of
Cheeger-Colding~\cite{CC97, CC00, CC00_1} and
Cheeger-Colding-Tian~\cite{CCT02} on the Gromov-Hausdorff
limits of manifolds with Ricci curvature bounded from below. Roughly
speaking the consequence of this theory is that for suitable families
of manifolds $(M_i, \omega_i)$, there is a fixed scale, independent of
$i$, at which a neighborhood of each point is well approximated by a
cone $C(Y)\times \mathbf{C}^{n-k}$, where $C(Y)$ is the metric
cone over a $(2k-1)$-dimensional metric space $Y$. Under suitable
conditions one can glue a section $\mathbf{1}$ with exponential decay
over such cones onto the manifolds $M_i$, to obtain sections
which are bounded away from zero at any given point.

This approach was first used by Donaldson-Sun~\cite{DS12} in the
context of a family of K\"ahler-Einstein manifolds, proving the result
analogous to Theorem~\ref{thm:main} in this case (see
Phong-Song-Sturm~\cite{PSS12} for an extension of these ideas to
K\"ahler-Ricci solitons). In the work of
Chen-Donaldson-Sun~\cite{CDS12} on the
Yau-Tian-Donaldson conjecture, a key role is
played by the extension of this result to K\"ahler-Einstein manifolds
which have conical singularities along a divisor.

In Section~\ref{sec:background} we review some of the results from the
theory of Cheeger-Colding, and Chen-Donaldson-Sun that we will use. In
Section~\ref{sec:lessthan1} we will focus on obtaining estimates for
solutions of \eqref{eq:Aubincont} for $t < T$, with $T < 1$. This
corresponds to Chen-Donaldson-Sun's paper \cite{CDS13_2} on conical
metrics with the limiting cone angle being less than $2\pi$. Finally
in Section~\ref{sec:Tis1}, we deal with the case $T=1$ in parallel
with \cite{CDS13_3} on the case when the cone angle tends to $2\pi$.

\subsection*{Acknowledgements}
I would like to thank Aaron Naber, Valentino Tosatti, and Ben
Weinkove
for useful discussions, and Kewei Zhang for pointing out a mistake in
an earlier version of the paper. In addition I thank Simon Donaldson for his interest in
this work, and for several useful comments on an earlier version of
this paper. The
author was supported in part by NSF grant DMS-1306298.

\section{Background} \label{sec:background}
We first give a very quick review of the convergence theory for manifolds
with Ricci curvature bounded from below. Recall that the
Gromov-Hausdorff distance $d_{GH}(X,Y)$ between two compact metric
spaces is the infimum of all $\delta > 0$, such that there is a metric
on the disjoint union $X\sqcup Y$ extending the metrics on $X,Y$ such
that both $X$ and $Y$ are $\delta$-dense. If a sequence of compact
metric spaces $X_i$ converge to $X_\infty$ in the Gromov-Hausdorff
sense, then we will always assume that we have chosen metrics $d_i$ on
$X_i\sqcup X_\infty$ extending the given metrics, such that both $X_i$
and $X_\infty$ are $\delta_i$-dense, with $\delta_i\to 0$.

Suppose now that we have a sequence of $n$-dimensional K\"ahler manifolds
$(M_i,\omega_i)$ satisfying
\begin{enumerate}
\item $\mathrm{Ric}(\omega_i) \geqslant 0$,
\item a non-collapsing condition $\mathrm{V}(B(p_i,1)) > c > 0$, and
\item a uniform diameter bound $\mathrm{diam}(M_i,\omega_i) < C$.
\end{enumerate}
From Gromov's compactness
theorem~\cite{Gro07} it follows that up to passing to a subsequence,
the $(M_i, \omega_i)$ converge to a compact metric space $(Z,
d_Z)$.

For each $p\in Z$, and any sequence $r_k\to 0$,
we can consider the sequence of pointed metric spaces $(Z, p,
r_k^{-1}d_Z)$. The pointed version of Gromov's theorem implies
that up to choosing a subsequence we can extract a limit $(Z_p,
p_\infty, d_\infty)$, which by the results of
Cheeger-Colding~\cite{CC97, CC00, CC00_1} is a metric cone $C(Y)$ over
a $2n-1$-dimensional metric space $Y$. It is called a tangent cone to
$Z$ at $p$, and in general such tangent cones are not unique. The
points in $Z$ can be classified according to how ``singular'' their
tangent cone is, giving rise to a stratification (we use the
convention of \cite{CDS13_2} for the subscripts):
\[ S_n \subset S_{n-1}\subset\ldots \subset S_1 = S \subset Z. \]
Here $S_k$ consists of those points where there is no tangent cone of
the form $\mathbf{C}^{n-k+1}\times C(Y)$. For instance at points in
$S_1\setminus S_2$ there must exist a tangent cone of the form $\mathbf{C}^{n-1}\times
\mathbf{C}_\gamma$, where $\mathbf{C}_\gamma$ is the flat
two-dimensional cone with cone angle $2\pi\gamma$. A key result due to
Cheeger-Colding is that the Hausdorff codimension of $S_k$ is at least
$2k$. The fact that in the K\"ahler case only even dimensional cones
occur, and complex Euclidean factors can be split off, is due to
Cheeger-Colding-Tian~\cite{CCT02}. 

The regular part $Z_{reg}\subset Z$ is defined to be $Z\setminus S$, and
at each point of $Z_{reg}$ every tangent cone is $\mathbf{C}^n$. This
follows from results of Colding~\cite{Col97}, which in effect say that
in the presence of non-negative Ricci curvature, the
Gromov-Hausdorff distance of a unit ball to the Euclidean unit ball is
comparable to the difference in volumes of the two balls. 
An important improvement of this is obtained in the presence of
an upper bound on $\mathrm{Ric}(\omega)$,
using a result due to Anderson~\cite{An90} (see
Theorem 10.25 in Cheeger~\cite{Ch01}).
\begin{prop}\label{prop:Anderson}
  There are constants $\delta, \theta > 0$ depending on $K> 0$
  with the following
  property. Suppose that $B(p,1)$ is a unit ball in a Riemannian
  manifold with bounds $0\leqslant \mathrm{Ric}(g) \leqslant
  Kg$ on its Ricci curvature.

  If $d_{GH}(B(p,1), B^{2n}) <
  \delta$, then for each $q\in B(p,\frac{1}{2})$, the ball $B(q,
  \theta)$ is the domain of a harmonic coordinate system in which the
  components of $\omega$ satisfy
  \[ \label{eq:harmonicprop}
     \begin{aligned} \frac{1}{2}\delta_{jk} &< g_{jk} <
    2\delta_{jk} \\
    \Vert g_{jk}\Vert_{L^{2,p}} &< 2, \text{ for all $p$. }
  \end{aligned}\]
\end{prop}

In the K\"ahler setting, the complex structure $J$ can also be assumed to
be close to the Euclidean complex structure $J_0$ in $C^{1,\alpha}$, using
that it is covariant constant. For sufficiently small $\delta$ one can
then construct holomorphic coordinates with respect to $J$ which are
close to the Euclidean coordinates in $C^{2,\alpha}$, using the
families version of Newlander-Nirenberg's interability
result~\cite{NW63}. It follows that in the K\"ahler case we can assume
that we have holomorphic coordinates on the ball $B(q,\theta)$
satisfying the properties \eqref{eq:harmonicprop} above, replacing
$\theta$ by a smaller constant if necessary. The fact that we obtain
holomorphic coordinates on a ball of a uniform size (once the complex
structure $J$ is sufficiently close to $J_0$) follows from the
method of proof in \cite{NW63}, and is made more explicit in the
sharper results of Hill-Taylor~\cite{HT07}, see for instance
Proposition 4.1.

It follows from this result that in the setting above, where
$(M_i, \omega_i) \to (Z,d_Z)$ in the Gromov-Hausdorff sense, in the
presence of a 2-sided Ricci curvature bound we have:
\begin{itemize}
  \item the regular set $Z_{reg}\subset Z$ is open, and
  \item the metrics $\omega_i$ converge in $C^{1,\alpha}$, uniformly
    on compact sets, to a K\"ahler
metric $\omega_Z$ on $Z_{reg}$ inducing the distance $d_Z$. 
\end{itemize}

\noindent The same ideas also apply to scaled limit spaces, so in particular to
tangent cones. 

One of the key
points in \cite{CDS13_2,CDS13_3} is to show that in their situation,
even though there is not a two-sided Ricci curvature bound, 
the regular set is still open, and one
obtains $L^p$ convergence of the metrics on the regular
part of limit spaces, for all $p$. This is also what our main goal is
going to be. 

We now review the construction
in Donaldson-Sun~\cite{DS12} that we will need. First recall the
following definition.
\begin{defn}\label{defn:goodcone}
  Suppose that $(Z,d_Z)$ is a Gromov-Hausdorff limit of K\"ahler manifolds as
  above, and $C(Y)$ is a tangent cone at $p\in Z$. We say that the
  tangent cone $C(Y)$ is good, if the following hold:
  \begin{enumerate}
    \item the regular set $Y_{reg}$ is open in
      $Y$, and smooth,
      \item the distance function on $C(Y_{reg})$ is induced by a
        Ricci flat K\"ahler metric,
      \item for all $\eta > 0$
  there is a Lipschitz function $g$ on $Y$, equal to 1 on a
  neighborhood of the singular set $S_Y\subset Y$, supported on the
  $\eta$-neighborhood of $S_Y$ and with $\Vert\nabla g\Vert_{L^2}
  \leqslant \eta$.
  \end{enumerate}
\end{defn}

With these definitions the following is the consequence of the main
construction in \cite{DS12} (see also 
\cite[Proposition 7]{CDS13_2}).  
\begin{prop}
  Suppose that $(M_i,\omega_i)$ are as above (with non-negative Ricci
  curvature, non-collapsed and with bounded diameter), and they
  converge to a limit space $(Z, d_Z)$. Suppose that each tangent cone
  to $Z$ is good, and that the metrics (after rescaling)
  converge in $L^p$ on compact sets of
  the regular part of the tangent cones (for some $p > 2n$) to the
  Ricci flat K\"ahler metric.

  Then for each $p\in Z$ there are $b(p), r(p) > 0$ and an integer
  $k(p)$ with the following property. There is a $k \leqslant k(p)$
  such that for sufficiently large $i$ there is a holomorphic section
  $s$ of $K_{M_i}^{-k}$ with $L^2$-norm 1, and with $|s(x)| \geqslant
  b(p)$ for all $x\in M_i$ with $d_i(p,x) < r(p)$. 
\end{prop}

A compactness and contradiction argument in \cite{DS12} then implies
the lower bound for the Bergman kernel that we need in Theorem~\ref{thm:main}. 
With this said, our main goal is to prove the following.
\begin{theorem}\label{thm:goodcones}
  Let $(M, \omega_t)$ be solutions of Equation~\ref{eq:Aubincont} for
  $t\in [0, T)$, with $T\leqslant 1$. Suppose that a sequence
  $(M,\omega_{t_i})$ converges to a limit space $(Z, d_Z)$. Then each
  tangent cone of $Z$ is good, and the metrics converge in $L^p$ on
  compact subsets of the regular part of each tangent cone, for all
  $p$. 
\end{theorem}

As in \cite{CDS13_2, CDS13_3} the problem is somewhat different in the
cases when $T < 1$ and $T=1$, and we will deal with these separately
in Section~\ref{sec:lessthan1} and Section~\ref{sec:Tis1}. 

An important part of our approach is to treat $\alpha$ in
\eqref{eq:Aubincont} as a K\"ahler metric as well, and use that
$\alpha$ is fixed. In a holomorphic chart we will get control of
$\alpha$ through the following $\epsilon$-regularity result
(see Ruan~\cite{Ru99}).
\begin{prop}\label{prop:epsilonreg}
  Suppose that $\alpha$ is a K\"ahler metric on an open subset
  $U\in\mathbf{C}^n$ such that $|\mathrm{Rm}(\alpha)| < K$. There are
  constants $\epsilon_0, C > 0$
  depending on $K$, such that if $B^{Euc}_r(x)\subset U$
  and
  \[\label{eq:smallenergy}
  r^{2-2n}\int_{B^{Euc}_r(x)} \alpha\wedge \omega_{Euc}^{n-1} <
  \epsilon_0, \]
  then
  \[ \sup_{B^{Euc}_{r/2}(x)} \mathrm{tr}(\alpha) < Cr^{-2n}\int_{B^{Euc}_r(x)}
  \alpha\wedge\omega_{Euc}^{n-1} < C\frac{\epsilon_0}{r^2}, \]
  where $\mathrm{tr}$ denotes the trace with respect to
  $\omega_{Euc}$. 
\end{prop}
\begin{proof}
  This follow from the $\epsilon$-regularity for harmonic maps
  (Schoen-Uhlenbeck~\cite{SU82}),
  applied to the identity map $(U, \omega_{Euc})\to (U,
  \alpha)$. The result also follows from the argument in
  Proposition~\ref{prop:epsilonreg2}. Indeed when $B$ is a Euclidean
  ball, then the hypotheses are all satisfied, using also the
  monotonicity of $V(q,r)$ in that case. 
\end{proof}

Following Chen-Donaldson-Sun~\cite{CDS13_3} we introduce the invariant 
\[ I(\Omega) = \inf_{B(x,r)\subset \Omega} VR(x,r) \]
for any domain $\Omega$, where $VR(x,r)$ is the ratio of the volumes
of the ball $B(x,r)$ and the Euclidean ball $rB^{2n}$. Note that by
Colding's volume convergence result, together with Bishop-Gromov
volume comparison, if $B$ is any unit ball in a manifold with non-negative
Ricci curvature, then $1-I(B)$ is controlled by the Gromov-Hausdorff
distance $d_{GH}(B,B^{2n})$, and it controls
$d_{GH}(\rho B, \rho B^{2n})$ for any $\rho < 1$. We will
need to obtain several variants of Proposition~\ref{prop:epsilonreg},
where the crucial monotonicity property of the quantity in
\eqref{eq:smallenergy} is missing. Instead, the monotonicity of
$I(\Omega)$ will be used.

\section{The case $T < 1$}\label{sec:lessthan1}
Suppose that $\alpha$ is a fixed K\"ahler form on $M$, and $0 < T_0 <
T_1 < 1$. Suppose that for some $t\in (T_0, T_1)$ we have 
\[ \mathrm{Ric}(\omega_t) = t\omega_t + (1-t)\alpha. \]
To study the structure of the Gromov-Hausdorff limit of a sequence of
such $\omega_{t_i}$ we need to study small balls with respect to these
metrics, scaled to unit size. A scaling $\widetilde{\omega}_t =
\Lambda\omega_t$  satisfies
\[ \mathrm{Ric}(\widetilde{\omega}_t) =
\Lambda^{-1}t\widetilde{\omega_t} + (1-t)\alpha. \]

Because of this, in 
this section we work with a unit ball $B=B(p,1)$ in a K\"ahler manifold
with metric $\omega$, such that
\[ \label{eq:Ricomega} \mathrm{Ric}(\omega) = \lambda\omega + \alpha, \]
where $\lambda\in (0,1]$ and $\alpha$ is a K\"ahler form on $B(p,1)$
which we control in the following sense. There is a number $K > 0$,
such that on each $\alpha$-ball $B_\alpha \subset B(p,1)$ of radius at
most $K^{-1}$ (in the $\alpha$-metric)
we have holomorphic coordinates, relative to which  
\[ \label{eq:alphabounds}\begin{gathered}
  \frac{1}{2}\delta_{j\bar k} < \alpha_{j\bar k} < 2\delta_{j\bar k}
  \\ 
  \Vert \alpha_{j\bar k} \Vert_{C^2} < K. 
\end{gathered} \] 
In addition we can assume that $B(p,1)$ is non-collapsed,
i.e. $\mathrm{V}(B(p,1)) > c > 0$. Note that this ball $B$ is a scaled
up version of a small ball with respect to a metric $\omega_t$ along
the continuity method. We obtain the above estimates for any such ball
with $c, K$ depending on $T_0, T_1$, as long as
$t\in (T_0, T_1)$.

One of the key new results is the following, which essentially
shows that at points in the regular set $Z_{reg}$ we are in the
setting of bounded Ricci curvature. The analogous result for conical
K\"ahler-Einstein metrics, \cite[Proposition 3]{CDS13_2}, is
straightforward because $I(B)$ is bounded away from 1 by a definite
amount for balls centered at a conical singularity.  The proof of the
result is reminiscent of arguments in Carron~\cite{Car10}. 
\begin{prop}\label{prop:Ricbound}
There is a $\delta > 0$, depending on $K$ above,
 such that if $1 - I(B) < \delta$, then 
\[ |\mathrm{Ric}(\omega)| < 5, \text{ on } \frac{1}{2}B. \]
\end{prop}
\begin{proof}
  We will bound $\alpha$, which is equivalent to bounding
  $\mathrm{Ric}(\omega)$.  Suppose that
 \[ \sup_B d_x^2 |\alpha(x)|_\omega = M, \] 
 where $d_x$ denotes the distance to the boundary of the ball, 
  and suppose that the supremum is achieved at $q\in B$. If $M > 1$,
  we can consider the ball
  \[ B\left(q, \frac{1}{2}d_qM^{-1/2}\right), \]
  scaled to unit size $\widetilde{B}$ with metric
  $\widetilde{\omega} = 4Md_q^{-2}\omega$. By construction, we have
\[ \begin{gathered}
          |\alpha|_{\widetilde{\omega}} \leqslant 1 \text{ on }
          \widetilde{B}, \\
          |\alpha(q)|_{\widetilde{\omega}} = \frac{1}{4}.  
\end{gathered}\]
In particular we have $\alpha\leqslant
\widetilde{\omega}$ on $\widetilde{B}$, and using the equation
satisfied by $\widetilde{\omega}$ we have
$|\mathrm{Ric}(\widetilde{\omega})| < 2$ on $\widetilde{B}$. At the
same time we have a unit vector
$v$ at $q$ (with respect to $\widetilde{\omega}$), such that
\[ \alpha(v,\bar{v}) \geqslant \frac{1}{16 n}, \]
otherwise the norm of $\alpha(q)$ would be too small. 

From the bound on $\mathrm{Ric}(\widetilde{\omega})$ on
$\widetilde{B}$, and Anderson's result, we have holomorphic
coordinates $\{z^1,\ldots,z^n\}$ on the ball $\theta\widetilde{B}$,
with respect to which the components of $\widetilde{\omega}$ are
controlled in $C^{1,\alpha}$. We can assume that $v=\partial_{z^1}$ at
$q$.  We can also assume without loss of
generality that $K^{-1} < \theta$, so we have good holomorphic
coordinates $\{w^1,\ldots,w^n\}$ for $\alpha$ on
$K^{-1}\widetilde{B}$ (meaning that the components of $\alpha$ are
controlled in $C^2$ as in \eqref{eq:alphabounds}),
since this is contained in a $K^{-1}$-ball with
respect to $\alpha$. The $w^i$ are holomorphic functions of the $z^i$,
and $|w_i| < 2K^{-1}$, so we obtain bounds on the derivatives
$\partial w^i/ \partial z^j$ in $\frac{1}{2}K^{-1}\widetilde{B}$. The upshot is that we can find a ball of a
definite size $\rho\widetilde{B}$, on which
\[ \alpha(\xi,\overline{\xi}) \geqslant \frac{1}{40n} \]
for any unit vector $\xi$ with angle $\angle(\xi,\partial_{z^1})
< \pi / 4$ (let us identify $T^{1,0}M$ with $TM$ in the usual way).
Making $\rho$ smaller if necessary, 
we can also assume that on $\rho\widetilde{B}$, if
$\gamma(t)$ is a unit speed geodesic with $\gamma(0)=0$ and
$\angle(\dot{\gamma}(0), \partial_{z^1}) < \pi /8$, then
$\angle(\dot{\gamma}(t), \partial_{z^1}) < \pi / 4$ for $t <
\rho$. Indeed, we have
\[ \frac{d}{dt} \langle \dot\gamma(t), \partial_{z^1}\rangle = \langle
\dot\gamma(t), \nabla_{\dot\gamma(t)} \partial_{z^1}\rangle, \]
and the covariant derivative of $\partial_{z^1}$ is bounded since in
the $z^i$ coordinates the components of $\widetilde{\omega}$ are
controlled in $C^{1,\alpha}$. It follows that we have a uniform bound
on the derivative of the angle
$\angle(\dot{\gamma}(t), \partial_{z^1})$ along the geodesic $\gamma$.

Since $\mathrm{Ric}(\widetilde{\omega}) > \alpha$,
this gives a positive lower bound on the radial component of
$\mathrm{Ric}(\widetilde{\omega})$ in these directions, and 
we can apply Bishop-Gromov volume comparison to the
corresponding spherical sector in $\rho\widetilde{B}$ (in the relevant
Bochner formula, the lower bound on the Ricci curvature is only
required in the radial directions). 
From this we obtain that the volume ratio
\[ VR(\rho\widetilde{B}) < 1 - \epsilon \]
for some small (but definite) $\epsilon > 0$, which is a contradiction
if we choose $\delta$ small enough in our assumptions. It follows that
if $\delta$ is sufficiently small, then $M \leqslant 1$, which implies
that $|\alpha|_\omega < 4$ in $\frac{1}{2}B$, and so $\alpha <
4\omega$ there. It follows that $\mathrm{Ric}(\omega) < 5\omega$ on
$\frac{1}{2}B$.  
\end{proof}

This proposition, together with Anderson's result, implies the 
following.  
\begin{prop}\label{prop:convergetoB}
Suppose that we have a sequence of unit balls $B(p_i,1)$, with
$\omega_i, \alpha_i$ satisfying the same conditions as $\omega,
\alpha$ above. Suppose that the $B(p_i,1)$, with the metrics
$\omega_i$ converge to the
Euclidean ball $B^{2n}$ in the Gromov-Hausdorff sense. Then the
convergence is $C^{1,\alpha}$ on compact subsets. 
\end{prop}

\begin{prop}
  If $B(p_i,1)\to Z$ in the Gromov-Hausdorff sense,
  then the regular set in $Z$ is
   open, and the convergence on the regular set is locally $C^{1,\alpha}$. 
\end{prop}

Next we examine the situation when $B_i(p_i,1)\to Z$, with $p_i\to p$,
and a tangent cone at $p\in Z$ is of the form
$\mathbf{C}_\gamma\times\mathbf{C}^{n-1}$, where $\mathbf{C}_\gamma$
denotes the flat cone with cone angle $2\pi\gamma$.  It follows from the previous
proposition that a sequence of scaled balls converges to
$\mathbf{C}_\gamma\times\mathbf{C}^{n-1}$ in $C^{1,\alpha}$, uniformly on
compact sets away from the singular set. The results of
Chen-Donaldson-Sun~\cite{CDS13_2} about good tangent cones can
therefore be applied. In particular the arguments in
\cite[Section 2.5]{CDS13_2} imply that for sufficiently large $k$,
with $i$ large enough (depending on $k$),  we can
regard $k\omega_i$ as a metric on the unit ball
$B^{2n}\subset\mathbf{C}^n$. We use co-ordinates $(u, v_1,\ldots,
v_{n-1})$, and let $\eta_\gamma$ be the conical metric
\[ \eta_\gamma = \sqrt{-1}\frac{du\wedge
  d\overline{u}}{|u|^{2-2\gamma}} + \sqrt{-1}\sum_{i=1}^{n-1}
dv_i\wedge d\overline{v}_i. \]
The metric $\widetilde{\omega}_i = k\omega_i$ then satisfies, for some
fixed constant $C$:
\begin{enumerate}
\item 
  $\widetilde{\omega}_i = \ddb\phi_i$ with $0\leqslant\phi_i\leqslant
  C$,
\item 
  $\omega_{Euc} < C\widetilde{\omega}_i$,
\item
  Given $\delta$ and a compact set $K$ away from $\{u=0\}$, we can
  suppose (by taking $i$ large once $k$ is chosen sufficiently large)
   that $|\widetilde{\omega}_i - \eta_\gamma|_{C^{1,\alpha}}
  < \delta$ on $K$. 
\end{enumerate}

\begin{prop}\label{prop:conebound}
  There are $0 < \gamma_1 < \gamma_2 < 1$ with the following property. If
  $B(p_i,1)\to Z$, with $p_i\to p$, and a tangent cone at $p$
  is $\mathbf{C}_\gamma\times
  \mathbf{C}^{n-1}$, then $\gamma\in (\gamma_1,\gamma_2)$. 
\end{prop}
\begin{proof}
  The existence of $\gamma_1$ follows from the volume non-collapsing
  assumption. For the upper bound we use
  Proposition~\ref{prop:Ricbound}
  and the fact
  that the unit balls in $\mathbf{C}_\gamma\times\mathbf{C}^{n-1}$
  converge to the Euclidean ball in the Gromov-Hausdorff sense as
  $\gamma\to 1$. In particular if $B(p_i,1)\to Z$ and a tangent
  cone at $p\in Z$ is $\mathbf{C}_\gamma\times
  \mathbf{C}^{n-1}$ for $\gamma$ sufficiently close to $1$, then
  for large $i$ we can find small balls around $p_i$, which when
  scaled to unit size are sufficiently close to the Euclidean ball for
  Proposition~\ref{prop:Ricbound} to imply that the Ricci curvature is
  bounded. In this case, however, there cannot be singularities of the
  form $\mathbf{C}_\gamma\times\mathbf{C}^{n-1}$ in the
  Gromov-Hausdorff limit according to
  Cheeger-Colding-Tian~\cite{CCT02}. This is a contradiction. 
\end{proof}

\begin{prop}\label{prop:densitylower}
  Let us denote by $\underline{0}$ the vertex in the cone
  $\mathbf{C}_\gamma\times \mathbf{C}^{n-1}$ for some $\gamma\in
  (\gamma_1,\gamma_2)$. There are constants $c_0, \delta > 0$ such
  that if $d_{GH}(B(p,1), B(\underline{0}, 1)) < \delta$, then 
  \[ \int_{B(p,1)} \alpha\wedge \omega^{n-1} > c_0. \]
\end{prop}
\begin{proof}
  We argue by contradiction. Suppose that $B(p_i, 1)$ is a sequence of
  unit balls in K\"ahler manifolds as above with corresponding forms
  $\alpha_i$, such that
  \[ \label{eq:dGH1} d_{GH}(B(p_i, 1), B(\underline{0}, 1)) \to 0. \]
  From the discussion above, once $i$ is sufficiently large, then
  for a fixed small $r_0 > 0$ we can view the scaled up metrics
  $r_0^{-2}\omega_i$ as  metrics on $B^{2n}$ such that
  \[ \label{eq:omegailower} \omega_{Euc} < Cr_0^{-2}\omega_i. \]
If, with $\epsilon_0$ from Proposition~\ref{prop:epsilonreg}
\[ \int_{B^{2n}} \alpha_i \wedge \omega_{Euc}^{n-1} < \epsilon_0, \]
then $\alpha_i < C'\omega_{Euc} < C'Cr_0^{-2}\omega_i$ on
$\frac{1}{2}B^{2n}$.  In this case $r_0^{-2}\omega_i$ has bounded Ricci
curvature, and so by \cite{CCT02}, no codimension two conical
singularity can form in the limit. This contradicts \eqref{eq:dGH1}
and so we must have  
\[ \int_{B^{2n}} \alpha_i \wedge \omega_{Euc}^{n-1} \geq
\epsilon_0 \]
for large $i$. Together with \eqref{eq:omegailower} this gives the
required result. 
\end{proof}

\begin{remark} In a previous version of this paper a stronger form of
  this proposition was stated, where the energy of $\alpha$ on the
  unit ball $B(p,1)$ was bounded below even if only a much smaller
  ball was close to a ball in $\mathbf{C}_\gamma\times
  \mathbf{C}^{n-1}$. It was pointed out to us by Kewei Zhang that the
  proof of that stronger result was incomplete, and instead we obtain
  just the above weaker statement. For this reason we also had to
  modify the proof of Proposition 14 below.
\end{remark}

\begin{prop}\label{prop:densityupper}
  There is a constant $A > 0$, such that 
  if $B(p,1)$ is sufficiently close to either the Euclidean unit ball
  or the unit ball in 
  a cone $\mathbf{C}_\gamma\times\mathbf{C}^{n-1}$ with
  $\gamma\in(\gamma_1,\gamma_2)$, 
  then 
  \[ \label{eq:intBpA} \int_{B(p,\frac{1}{2})} \alpha\wedge\omega^{n-1} < A. \]
\end{prop}
\begin{proof}
  For the case of the Euclidean ball this follows from
  Proposition~\ref{prop:Ricbound}. For the cones, 
  using Gromov compactness, it is enough to get a bound $A$ for a
  fixed cone angle $\gamma$, and this will imply a uniform bound for
  $\gamma\in (\gamma_1,\gamma_2)$, and
  even for $\gamma\in(\gamma_1,1]$.
  
  Suppose that $B(p,1)$ is very close to the unit ball in
  $\mathbf{C}_\gamma\times \mathbf{C}^{n-1}$. 
  After a further scaling by a
  fixed factor, we can assume that we are in the situation described
  before Proposition~\ref{prop:conebound}, and so $\omega$ can be
  thought of as a metric on the Euclidean ball. Because of the
  scaling, this will only bound
  the integral in \eqref{eq:intBpA} on a smaller ball, but we can
  cover $B(p,\frac{1}{2})$ with balls each of which is, up to scaling, very
  close to the unit ball in the cone, or the Euclidean unit
  ball. Adding up the contributions we will obtain
  \eqref{eq:intBpA}. 
  
  Let us write
  \[  G = \frac{\omega^n}{\omega_{Euc}^n}, \]
  so that
  \[\begin{aligned} \mathrm{Ric}(\omega) &= -\ddb \log G, \\
     \mathrm{Ric}(\omega)\wedge \omega_{Euc}^{n-1} &= -\frac{1}{n}
     \Delta\log G\,\omega_{Euc}^n,
\end{aligned}\]
  with $\Delta$ denoting the Euclidean Laplacian. We then have 
  \[ \label{eq:intR} \begin{aligned}
  \int_{B_r} \mathrm{Ric}(\omega)\wedge\omega_{Euc}^{n-1} &=
  -\frac{1}{n}\int_{\partial B_r} \nabla_n 
  \log G\, dS \\
&= - \frac{1}{n}V(\partial B_r) \dashint_{\partial B_r} \nabla_n
  \log G\,dS \\
&= -\frac{1}{n} V(\partial B_r) \frac{d}{dr} \dashint_{\partial B_r}
  \log G\,dS,
\end{aligned} \]
where everything is computed using the metric $\omega_{Euc}$, and
$B_r$ denotes the Euclidean ball of radius $r$. In particular
the average 
$\dashint_{\partial B_r} \log G\,dS$
is decreasing with $r$. 

There is an $\epsilon > 0$ depending on the lower bound on the cone
angle $\gamma$, such that the Euclidean ball $B_\epsilon$ is
contained in the ball of radius $1/20$ in
the conical metric. By volume convergence we have an upper bound
on the volume of this ball with respect to $\omega$. In sum we have
\[ \label{eq:Gaverage}
\dashint_{B_\epsilon} G\,\omega_{Euc}^n < C_1, \]
so there is a point $q\in B_\epsilon$ satisfying $G(q) <C_1$. 
By translating our ball slightly,
we can assume that $q=0$. Define now the function
\[ f(r) = \dashint_{\partial B_r} \log G\,dS. \]
We know that $f(r)$ is decreasing, and $f(0) = \log G(0) < C_1$.

Using that $\omega_{Euc} < C\omega$, we have a lower bound on $G$, and
so, using volume convergence as well, 
\[ \label{eq:logGint}
\int_{B_{3/4}\setminus B_{1/2}} \log G\,\omega^n > -C_2, \]
for some constant $C_2$.  
We have
\[ \begin{aligned}
  -C_2 < \int_{B_{3/4}\setminus B_{1/2}} \log G\,\omega^n &= \int_{1/2}^{3/4} V(\partial
  B_r)\dashint_{\partial B_r} \log G\,dS\,dr \\
  &= \int_{1/2}^{3/4} V(\partial B_r) f(r)\, dr.
  \end{aligned}\]
  In this range of $r$ we have $c_0 < V(\partial B_r) < c_1$ for
  some fixed numbers $c_0,c_1$, and the fact that $f(0) <
  C$ and that $f$ is decreasing implies that
  $f(\frac{1}{2}) < C$. It is then clear that we must have
  some $r\in \left(\frac{1}{2}, \frac{3}{4}\right)$ where we have a
  lower bound for $f'(r)$, and this, together with \eqref{eq:intR}
  implies
  \[ \int_{B(p,\frac{1}{2})} \alpha\wedge \omega_{Euc}^{n-1} < A. \]
  Finally to obtain \eqref{eq:intBpA} with $\omega_{Euc}$ replaced by
  $\omega$ we can use the Chern-Levine-Nirenberg argument as in the
  proof of in \cite[Proposition 15]{CDS13_2}, since we have a bound on
  the K\"ahler potential. 
  
\end{proof}

Suppose now that we have a sequence of balls $B_i$ satisfying our
hypotheses, and $B_i \to Z$ in the Gromov-Hausdorff sense. We can 
follow the arguments in Chen-Donaldson-Sun~\cite{CDS13_2}, Section
2.8, to show that all tangent cones of $Z$ are good. As in
\cite{CDS13_2} it is easier to explan the proof of a slightly
different result. Namely, denote
the singular set of $Z$ by
\[S(Z) = S_2(Z) \cup \mathcal{D},\]
where $S_2(Z)$ denotes the
complex codimension 2 singularities, and $\mathcal{D}$ denotes the
points that have a tangent cone of the form
$\mathbf{C}_\gamma\times\mathbf{C}^{n-1}$.
\begin{prop}
  For any compact set $K\subset Z$, the set $K\cap S(Z)$ has capacity
  zero, i.e. there is a cutoff function as in
  Definition~\ref{defn:goodcone} of good tangent cones. 
\end{prop}
\begin{proof}
  The proof is very similar to the one in \cite{CDS13_2}, except that
  our Proposition~\ref{prop:densitylower} is weaker than the
  corresponding Proposition 17 in \cite{CDS13_2}. 

Define, for $z\in Z$ and $0 < \rho < 1$
  \[ V(i,z,\rho) = \rho^{2-2n}\int_{\widetilde{B}_i(z,\rho)} \alpha_i
  \wedge\omega_i^{n-1}, \]
  where 
  \[ \widetilde{B}_i(z,\rho) = \{x\in B_i\,:\, d_i(x,z) < \rho\}, \]
  and as usual we fix a distance function $d_i$ on $Z\sqcup B_i$ realizing
  the Gromov-Hausdorff convergence. 
  By the Gromov-Hausdorff convergence, for each $\rho > 0$,
  the ``ball'' $\widetilde{B}_i(z,\rho)$ is comparable to a ball of
  radius $\rho$ in $B_i$ for sufficiently large $i$ (i.e. it is
  contained between balls of radius $\rho/2$ and $2\rho$).

If $x\in Z\setminus S_2(Z)$, then by
Proposition~\ref{prop:densityupper} there exists a $\rho_x > 0$ such
that $V(i,x,\rho_x) < A$ for all large $i$. In addition if $x\in
\mathcal{D}$, then using Proposition~\ref{prop:densitylower}, 
for all $\delta > 0$ there exists an $r_x < \delta$
such that $V(i,x, r_x) > c_0$ for all large $i$. Note that in contrast
with the situation in \cite{CDS13_2}, we might
not have this inequality for all sufficiently small $r_x$, 
but rather for each $x$ there
is a certain sequence of radii going to zero for which we have the
inequality. 

Let $K\subset Z$ be compact. By Cheeger-Colding's Hausdorff dimension estimate of
$S_2$, for any small $\epsilon > 0$
 we can cover $S_2\cap K$ with balls $B_\mu$ such that
\[ \sum_\mu r_\mu^{2n-3} < \epsilon. \]
The set 
\[ J = K \setminus \cup_\mu B_\mu \]
is compact and covered by the balls $B_{\rho_x}(x)$. We choose a
finite subcover corresponding to $x_1,\ldots, x_N$, and set
\[ W = \bigcup_{j=1}^N B_{\rho_{x_j}}(x_i). \]
For sufficiently large $i$ we then get an estimate 
\begin{equation}\label{eq:1}
\int_W \alpha_i \wedge \omega_i^{n-1} < C 
\end{equation}
for some uniform $C$ (depending on $\epsilon$ above). Note also that
$J\subset K\setminus S_2$. 

We claim that the compact set $\mathcal{D} \cap J$ has finite $(2n-2)$-dimensional
Hausdorff measure. To prove this, recall that for any small $\delta >
0$ and all $x\in
\mathcal{D}\cap J$ we have  $r_x < \delta$ such that $V(i,x,r_x) >
c_0$ for large $i$. By a Vitali type covering argument we can find a
disjoint sequence of balls $B_{r_{x_k}}(x_k)$ in $W$ such that
$B_{5r_{x_k}}(x_k)$ cover all of $\mathcal{D}\cap J$. It follows that
\[ \mathcal{H}^{2n-2}_\delta (\mathcal{D}\cap J) \leq \sum_k 5^{2n-2}
r_{x_k}^{2n-2}. \]
At the same time for each $x_k$, we have the estimate
\[ r_{x_k}^{2-2n} \int_{B_{r_{x_k}}(x_k)} \alpha_i \wedge
\omega_i^{n-1} > c_0, \]
for sufficiently large $i$, and so using \eqref{eq:1} we have
\[ \sum_{k=1}^M c_0 r_{x_k}^{2n-2} < C \]
for any $M$, for a constant $C$ independent of $M$. It follows that
\[ \sum_{k=1}^\infty c_0 r_{x_k}^{2n-2} \leq C, \]
and since $\delta$ was arbitrary (and $C$ is independent of $\delta$), 
this implies that  $\mathcal{H}^{2n-2}(\mathcal{D}\cap J) \leq C$. 

It follows that $\mathcal{D}\cap J$ has capacity zero, in the
sense that for any $\kappa > 0$ we can find a cutoff function
$\eta_1$ supported in the $\kappa$-neighborhood of $\mathcal{D}\cap
J$, such that $\Vert \nabla \eta_1\Vert_{L^2} \leq \kappa$, and
$\eta_1 = 1$ on a neighborhood $V$ of $\mathcal{D}\cap J$ (see for
instance \cite[Lemma 2.2]{Bou} or \cite[Theorem 3, p. 154]{EG92}). The set
$(K\cap S(Z)) \setminus V$ is compact, and so it is covered by
finitely many of our balls $B_\mu$ from before. Because of this, as in
\cite{DS12}, 
we can find a good cutoff function 
$\eta_2$, with $\Vert \nabla \eta_2\Vert_{L^2} \leq \kappa$ (if
$\epsilon$ at the beginning was sufficiently small) and with
$\eta_2 = 1$ on a neighborhood of $K\cap S(Z)\setminus V$. Then
$\eta_1 + \eta_2 \geq 1$ on $S(Z)\cap K$, and so $\eta_1 + \eta_2$ can
be truncated to give the required cutoff function. 
\end{proof}

As explained in \cite{CDS13_2} a very similar proof shows that all
tangent cones in a limit space $Z$ of a sequence of balls $B(p_i,
1)$ are good, proving Theorem~\ref{thm:goodcones} in the case $ T< 1$. 

\section{The case $T=1$}\label{sec:Tis1}
When $T=1$, then 
we need to study non-collapsed balls $B = B(p,1)$ with metrics
satisfying
\[\label{eq:cont2}
\mathrm{Ric}(\omega) = \lambda\omega + (1-t)\alpha, \]
where $t < 1$ and $\lambda\in [0,1]$. These will be small balls in
$(M,\omega_t)$ along the continuity method, scaled up to unit size. 
 We can still assume that $\alpha$ satisfies the bounds
\eqref{eq:alphabounds} in suitable holomorphic coordinates, on any
ball of radius $K^{-1}$ (measured
using the metric $\alpha$). The issue that arises when $T=1$
is that $(1-t)\alpha$ no longer satisfies such bounds as $t\to 1$. 

We follow the arguments in Chen-Donaldson-Sun~\cite{CDS13_3} and we
will point out the analogies with their results. One of the main
difficulties when $T=1$ is that we cannot control the integral of
$\alpha$ on $B(p,1)$ even  
when $B(p,1)$ is Gromov-Hausdorff close to the Euclidean ball (note
that $B(p,1)$ is a possibly very small ball in the original metric
along the continuity method scaled to unit size). When $T < 1$, we
could achieve this in Proposition~\ref{prop:Ricbound} using
essentially that $\alpha$ controlled the Ricci curvature from above
and
 below in that case. 

A crucial tool in \cite{CDS13_3} is their Proposition 1, which
does not make use of the conical singularity, and
applies just as well in our situation. First we recall some
definitions. For a subset $A$ in a $2n$-dimensional length space $P$,
and for $\eta < 1$, let $m(\eta, A)$ be the infimum of those $M$ for which
$A$ can be covered by $Mr^{2-2n}$ balls of radius $r$ for all $\eta
\leqslant r < 1$.

For $x\in B$ and $r,\delta > 0$ a holomorphic map $\Gamma : B(x,r) \to
\mathbf{C}^n$ is called an $(r,\delta)$-chart centered at $x$ if
\begin{itemize}
\item $\Gamma(x) =0$,
\item $\Gamma$ is a homeomorphism onto its image,
\item For all $x', x''\in B(x,r)$ we have
  $|d(x',x'') - d(\Gamma(x'), \Gamma(x''))| \leqslant \delta,$
\item For some fixed $p > 2n$, we have
  $ \Vert \Gamma_*(\omega) - \omega_{Euc}\Vert_{L^p} \leqslant
  \delta.$
\end{itemize}

With these definitions, \cite[Proposition 1]{CDS13_3} is the
following. 
\begin{prop}\label{prop:CDS}
  Given $M,c$ there are $\rho(M), \eta(M,c), \delta(M, c) > 0$ with
  the following effect. Suppose that $1-I(B) < \delta$ and $W\subset
  B$ is a subset with $m(\eta,W) < M$, such that for any $x\in
  B\setminus W$ there is a $(c\eta, \delta)$-chart centered at
  $x$. Then (if the constant $K$ in the properties of $\alpha$
  is large enough):
  \begin{enumerate}
  \item
    There is a holomorphic map $F:B(p,\rho)\to \mathbf{C}^n$ which is
    a homeomorphism to its image, $|\nabla F| < K$, and its image lies
    between $0.9\rho B^{2n}$ and $1.1\rho B^{2n}$.
  \item
    There is a local K\"ahler potential $\phi$ for $\omega$ on
    $B(p,\rho)$ with $|\phi|\rho^{-2} < K$.
  \end{enumerate}
\end{prop}

The results which we have to modify in \cite{CDS13_3} are their
Corollary 2, and Propositions 5, 6.
For any $B(q,r)\subset B$, let us define the ``volume density''
\[ V(q,r) = r^{2-2n}\int_{B(q,r)} \alpha\wedge \omega^{n-1}. \]
The following is the analog of \cite[Corollary 2]{CDS13_3}. 
\begin{prop}\label{prop:CDScor2}
  Given $M$, suppose that the ball $B$ satisfies the hypotheses of
  Proposition~\ref{prop:CDS} for some $c > 0$. There are
  $A, \kappa > 0$, depending on $M$, such that if 
\[ \label{eq:intkappa}
  \int_B \alpha \wedge \omega^{n-1} < \kappa, \]
  then $\alpha < A\kappa \omega$ on $\frac{1}{3}\rho B$, 
  where $\rho=\rho(M)$ from Proposition~\ref{prop:CDS}. 
\end{prop}
\begin{proof}
  From Proposition~\ref{prop:CDS} we know that we can think of the
  metric $\omega$ as a metric on the Euclidean ball $0.9\rho B^{2n}$,
  where we have $\omega_{Euc} < K\omega$. We also think of $\alpha$ as
  being defined on this ball, and then \eqref{eq:intkappa} implies that
\[ (0.9\rho)^{2-2n}\int_{0.9\rho B^{2n}} \alpha \wedge \omega_{Euc}^{n-1} <
C_1\kappa, \]
  for some $C_1$ (which depends on $\rho$, and thus on $M$). 
  Our bounds on $\alpha$ imply that its curvature is bounded, so the
  $\epsilon$-regularity, Proposition~\ref{prop:epsilonreg}, implies
  that once $\kappa$ is sufficiently small, we have
\[ \alpha < C_2 \kappa\, \omega_{Euc} < KC_2\kappa\,\omega 
\, \text{ on } 0.45\rho
B^{2n}, \]
  for some $C_2$ (depending on $\rho$). 
\end{proof}

Our next goal is to obtain a weak form of monotonicity of the
volume density (note that $V(q,r)$ is 
monotone in $r$ if $\omega$ is the Euclidean metric),
which is analogous to \cite[Proposition 5]{CDS13_3}. For this we
first need
the following variant of the $\epsilon$-regularity result,
Proposition~\ref{prop:epsilonreg}, which does not use monotonicity. 
\begin{prop}\label{prop:epsilonreg2}
  There are $\delta, \epsilon > 0$ with the following
  properties. Suppose that $1 - I(B) \leqslant \delta$, and 
 \[ \label{eq:ineqepsilon}
  \sup_{B(q,r)\subset B} V(q,r) < \epsilon \]
  Then $\alpha \leqslant 4\omega$ on
  $\frac{1}{2}B$. 
\end{prop}
\begin{proof}
  The proof is similar to the proof of
  Proposition~\ref{prop:Ricbound}. Suppose that 
  \[ \sup_B d_x^2 |\alpha|(x) = M, \]
  where $d_x$ is the distance to the boundary of $B$, and suppose that
  the supremum is achieved at $q\in B$. If $M > 1$, let
  $\widetilde{B}$ with metric $\widetilde{\omega}$ be the ball
  $ B(q, 0.5d_q M^{-1/2})$ scaled to unit size.  On $\widetilde{B}$ we
  have $\alpha \leqslant \widetilde{\omega}$, and at the same time, at
  the origin we have  
  \[ \alpha \wedge \widetilde{\omega}^{n-1}(0) \geqslant
  \frac{1}{4n}\widetilde{\omega}^n(0).\]
  In particular on $\widetilde{B}$ we have a two-sided Ricci bound, so
  Anderson's result gives good holomorphic coordinates on $\theta
  \widetilde{B}$ once $\delta$ is chosen small enough. By the same
  argument as in the proof of Proposition~\ref{prop:Ricbound} we find
  a ball $\rho \widetilde{B}$ of a definite size, on which 
  \[  \alpha \wedge \widetilde{\omega}^{n-1}\geqslant
  \frac{1}{10n}\widetilde{\omega}^n,\]
  and so
  \[ \rho^{2-2n}\int_{\rho\widetilde{B}} \alpha\wedge
  \widetilde{\omega}^{n-1} \geqslant c_1, \]
  where we also used the non-collapsing assumption (and $c_1$ depends
  on $\rho$, but $\rho$ is a fixed number). If $\epsilon$ is chosen
  sufficiently small, then this contradicts our assumption
  \eqref{eq:ineqepsilon}. We then must have $M\leqslant 1$ and so
  $\alpha\leqslant 4\omega$ on $\frac{1}{2}B$. 
\end{proof}

\begin{prop}\label{prop:CDSprop5}
Given the $\epsilon > 0$ from Proposition~\ref{prop:epsilonreg2}, 
we have $\delta, \kappa > 0$ with the following properties. Suppose that
$1-I(B)\leqslant \delta$. If $B(q,r) \subset B(p,1/2)$ and $V(q,r)
\geqslant \epsilon$, then $V(q,R) \geqslant \kappa$ whenever $ R > r$
and $B(q,R)\subset B(p,1/2)$. 
\end{prop}
\begin{proof}
  Note first that $V(q,r)\to 0$ as $r\to 0$ for all $q\in B$. This
  means that if $V(q,r) > \epsilon$ and $q\in B(p,1/2)$, then $r$ is
  bounded away from zero. Let us also fix $\rho\in (0,1)$ which we will
  choose later, and we will initially restrict attention to pairs $r,R$ with
  $r < \frac{\rho}{6} R$. We can find a
  point $q_0$ and $r_0 \leqslant \frac{\rho}{6} R_0$ with $B(q_0,R_0) \subset
  \overline{B}(p,1/2)$
   such that
  \[ V(q_0, r_0) = \epsilon, \]
  and $V(q_0, R_0)$ is minimal in the sense that
  $V(q, R) \geqslant V(q_0, R_0)$ for all $q, R$ for which
 \begin{itemize}
\item 
    $B(q,R)\subset B(p,1/2)$ 
\item and $V(q,r)=\epsilon$ for
    some $r < \frac{\rho}{6} R$.
 \end{itemize}
 Let $\widetilde{B}$ denote the ball $B(q_0, R_0)$ scaled to unit
 size, with metric $\widetilde{\omega}$, and denote
 \[ \kappa = V(q_0, R_0), \]
so
\[\label{eq:intalphakappa}
 \int_{\widetilde{B}} \alpha\wedge \widetilde{\omega}^{n-1}
=\kappa, \]
and we are trying to prove a lower bound for $\kappa$. 
We have
\[ \label{eq:monotone}
V(q,r) < \epsilon\,\text{ for all }r < \frac{\rho}{6} R, \text{ if }V(q,R) <
\kappa\text{ and }B(q,R)\subset  B(p,1/2). \]
  For any $0 < \eta < 1/2$, let
 \[ Z_\eta = \{ x\in \widetilde{B}\,:\,
 r^{2-2n}\int_{\widetilde{B}(x, r)} \alpha\wedge
 \widetilde{\omega}^{n-1} \geqslant 2^{2-2n}\kappa, \text{for all }r\in (\eta,
 1/2)\}. \]
 This set has two properties:
\begin{itemize}
\item Suppose that $q\in \widetilde{B}\setminus Z_\eta$, and $x\in
  B(q, \eta/2)$. There is a $\tau\in (\eta, 1/2)$ such that
 \[ \tau^{2-2n}\int_{\widetilde{B}(q,\tau)}
 \alpha\wedge\widetilde{\omega}^{n-1} < 2^{2-2n}\kappa, \]
 and so using $\widetilde{B}(x,\tau/2)\subset \widetilde{B}(q, \tau)$
 we get 
 \[ \left(\frac{\tau}{2}\right)^{2-2n}\int_{\widetilde{B}(x,\tau/2)}
 \alpha \wedge\widetilde{\omega}^{n-1} < \kappa. \]
 Using \eqref{eq:monotone} we then have
\[ r^{2-2n}\int_{\widetilde{B}(x, r)} \alpha\wedge
\widetilde{\omega}^{n-1} < \epsilon, \]
for all $r < \frac{\rho\eta}{12}$. In particular we can apply
Proposition~\ref{prop:epsilonreg2} to the ball $B(q, \rho\eta/12)$. 
It follows that we have 
$\alpha \leqslant 4\cdot (\rho\eta/12)^{-2}\widetilde{\omega}$ on
$\widetilde{B}(q,\rho\eta/24)$. Using Anderson's result we get a $(\theta\rho\eta,
\delta')$-chart centered at $q$, for some $\theta,\delta' > 0$, where we can make $\delta'$
arbitrarily small by choosing $\delta$ small enough. 
\item Using \eqref{eq:intalphakappa} we get that for any $r\in [\eta,
  1)$, the set $Z_\eta$ can
be covered by $Mr^{2-2n}$ balls of radius $r$ for a universal
constant $M$.  
\end{itemize}
Using the number $M$ in Proposition~\ref{prop:CDS} we obtain a
$\rho(M)$, which we fix as our choice of $\rho$. Feeding back $M$ and
$c=\theta\rho$ into  Proposition~\ref{prop:CDS} we get
numbers $\eta(M,c),\delta(M,c)$. We can choose our
$\delta$ so that it, and $\delta'$ are less than $\delta(M,c)$. The
second point above implies that we can use the set $W=Z_{\eta(M,c)}$ in
Proposition~\ref{prop:CDS}, and so Proposition~\ref{prop:CDScor2}
applies to $\widetilde{B}$. In particular, if $\kappa$ is sufficiently
small, then $\alpha < A\kappa\widetilde{\omega}$ on
$\frac{1}{3}\rho\widetilde{B}$, 
and so we have 
\[ r^{2-2n}\int_{\widetilde{B}(q_0, r)} \alpha\wedge
\widetilde{\omega}^{n-1} < c_n A\kappa r^2 \]
for a dimensional constant $c_n$ and all $r < \rho/3$. 
Translating back to the unscaled ball this means that
\[ V(q_0, r) < c_nA\kappa \frac{R_0^2r^2}{4}, \]
for all $r < (R_0\rho)/6$. If $\kappa$ is too small, then this
contradicts $V(q_0,r_0)=\epsilon$. 

We have shown that we have a (universal) $\rho$ with the following
property. If  
$B(q,r) \subset B(p,1/2)$ and $V(q,r)
\geqslant \epsilon$, then $V(q,R) \geqslant \kappa$ whenever 
$ \frac{\rho}{6}R > r$ and $B(q,R)\subset B(p,1/2)$. Assume now that
$V(q,R) < \kappa$, and $r < R$. If $r < \frac{\rho}{6}R$ then we know
that $V(q,r) < \epsilon$, while if $r > \frac{\rho}{6}R$ then we have
\[ V(q,r) < \left(\frac{R}{r}\right)^{2n-2} V(q,R) <
\left(\frac{6}{\rho}\right)^{2n-2}\kappa. \]
Choosing $\kappa$ sufficiently small we can therefore get $V(q,r) <
\epsilon$ for all $r < R$. 
\end{proof}

The following is analogous to \cite[Proposition 6]{CDS13_3}. 
\begin{prop}\label{prop:CDSprop6}
  Let $K$ be the number from Proposition~\ref{prop:CDS}. Given $A,
  \theta > 0$ there are $\sigma(K), \gamma(A,\theta), \delta_*(\theta)
  > 0$ with the following effect. Suppose that $\int_B
  \alpha\wedge\omega^{n-1} < A$ and that $1-I(B) <
  \delta_*(\theta)$. Suppose that there is a holomorphic
  $F:B\to\mathbf{C}^n$ with $F(p)=0$, which is a homeomorphism onto a
  domain between $B^{2n}$ and $1.1B^{2n}$. In addition assume $|\nabla
  F| < K$ and that $\omega$ has a K\"ahler potential with $|\phi|<
  K$. If $t > 1-\gamma$, then there is a $(\sigma,\theta)$-chart
  centered at $p$. 
\end{prop}
\begin{proof}
  Using the $\epsilon$-regularity result
  Proposition~\ref{prop:epsilonreg}, the proof is essentially the same
  as that in \cite{CDS13_3}. First, if $A$ is sufficiently small, then
  the $\epsilon$-regularity, as in the proof of
  Proposition~\ref{prop:CDScor2}
  implies that $\alpha$ is actually bounded,
  and we are in the situation of bounded Ricci curvature, so
  Anderson's result applies.

  For large $A$ we argue by contradiction, so we have a sequence of
  balls $(B_i, \omega_i)$ with additional K\"ahler forms $\alpha_i$,
  satisfying \eqref{eq:cont2}, with $t_i\to 1$. By assumption we can
  think of the $\omega_i$ and $\alpha_i$ as metrics on $B^{2n}$, with
  $\omega_i$ having bounded K\"ahler potential, and $\omega_{Euc} <
  K\omega_i$. We have
  \[ \label{eq:intalphai}
  \int_{B^{2n}} \alpha_i\wedge\omega_{Euc}^{n-1} <
  K^{n-1}\int_{B^{2n}}\alpha_i\wedge\omega^{n-1} < CA. \]
  It follows that up to choosing a subsequence, the forms $\alpha_i$
  converge weakly to a limiting current $\alpha_\infty$. By Siu's
  theorem~\cite{Siu74}, the set $V\subset B^{2n}$
  where the Lelong numbers of $\alpha_\infty$ are at least
  $\epsilon_0/2$ (with $\epsilon_0$ from
  Proposition~\ref{prop:epsilonreg})  is an
  analytic subset.  For any $p\in B^{2n}\setminus V$, there is a
  radius $r_p$, such that
  \[ r_p^{2-2n}\int_{B^{Euc}(p, r_p)} \alpha\wedge\omega_{Euc}^{n-1} \leqslant
  \epsilon_0/2, \]
  and so for sufficiently large $i$, the same inequality holds for
  $\alpha_i$ with $\epsilon_0/2$ replaced by
  $\epsilon_0$. Proposition~\ref{prop:epsilonreg} then implies that
  for $i > N_p$ we have
  \[ \label{eq:supalphai} \sup_{B_{r_p/2}^{Euc}}
  \mathrm{tr}_{\omega_{Euc}}
  \alpha_i < Cr_p^{-2},\]
  so
  for any compact subset of $B^{2n}\setminus V$ we can find a
  uniform bound on the Ricci curvature of the $\omega_i$, and so by
  Anderson's result~\cite{An90} (and using that the Euclidean $r$-ball
  contains the ball of radius
  $rK^{-1/2}$ in the metric $\omega_i$) the metrics $\omega_i$ converge on
  $B^{2n}\setminus V$ locally in
  $C^{1,\alpha}$ to a limit $\omega_\infty$.

  We now want to write $\alpha_i = \ddb f_i$ for all $i$, 
  including $i=\infty$, so that $f_i \to f_\infty$
  locally in $L^1$. We can do this by obtaining $f_i$ for finite $i$,
  through the usual proof of the local
  $\partial\overline{\partial}$-lemma as in
  Griffiths-Harris~\cite{GH78} for instance. First we find 1-forms
  $\beta_i$ such that $\alpha_i=d\beta_i$ using the proof of the
  usual Poincar\'e lemma, using a base point $p\in B^{2n}\setminus V$ in the
  argument. Since we control the $\alpha_i$ uniformly in a
  neighborhood of $p$, the $\beta_i$ will also be controlled there. At
  the same time the integral bound \eqref{eq:intalphai} implies
  uniform $L^1$-bounds for the coefficients of $\beta_i$ on
  $B^{2n}$. Now we use the proof of the
  $\overline{\partial}$-Poincar\'e lemma \cite[p. 25]{GH78} to obtain $h_i$ with
  $\overline{\partial} h_i = \beta_i^{0,1}$. From the uniform control
  of $\beta_i$ near $p$ and the integral bound on $B^{2n}$ it follows
  that we have uniform bounds on the $h_i$ in a neighborhood of
  $p$. We can then take $f_i = 2\mathrm{Im}(h_i)$. The $f_i$ are
  plurisubharmonic functions on $B^{2n}$, with uniform bounds on a
  neighborhood of $p$, so after taking a subsequence we can assume
  that $f_i \to f_\infty$ in $L^1_{loc}$, and consequently we 
  have $\alpha_\infty = \ddb f_\infty$. 

  On any compact set in $B^{2n}\setminus V$ we have uniform bounds on
  $\Delta f_i$ from \eqref{eq:supalphai} so on such compact sets we obtain
  $C^{1,\alpha}$ bounds on the $f_i$ (independent of $i$). Using
  this, the rest of the proof in \cite{CDS13_3} can be followed
  closely.

  Let us write $\omega_i = \ddb\phi_i$. The equation
  \eqref{eq:cont2} can be written as
  \[ \label{eq:ieq}
 \det(\partial_j\overline{\partial}_k\phi_i) = e^{-\lambda_i\phi_i - (1-t_i)f_i}|U_i|^2, \]
  where $\lambda_i \leqslant 1$ and
  $U_i$ is a nowhere vanishing holomorphic function on
  $0.9B^{2n}$. As in \cite{CDS13_3} we can bound $|U_i|$ from above and
  below uniformly on a smaller ball $0.8B^{2n}$, and so on this ball we get
  \[ \label{eq:upbound} K^{-1}\omega_{Euc} \leqslant \omega_i \leqslant
  C'e^{-(1-t_i)f_i}\omega_{Euc},\]
  for some constant $C'$. Note that by Demailly-Koll\'ar~\cite[Theorem 0.2]{DK01}, we have a
  constant $\kappa > 0$, such that $e^{-\kappa f_i} \to e^{-\kappa
    f_\infty}$ in $L^1$, over $0.8B^{2n}$. In particular for any $q >
  0$, once $i$ is sufficiently large, we have $(1-t_i)q < \kappa$, and
  so from \eqref{eq:upbound} 
we have a uniform bound on the $L^q$-norm of $\omega_i$. The $C^{1,\alpha}$
  convergence of $\omega_i$ to $\omega_\infty$ away from $V$
 then implies that the $\omega_i$ converge to
  $\omega_\infty$ in $L^p$ for any $p$. It follows that up to choosing
  a subsequence, the potentials $\phi_i$ for $\omega_i$ converge in
  $L^{2,p}$ to a potential $\phi_\infty$ for $\omega_\infty$. Up to
  choosing a further subsequence we can
  take the limit in \eqref{eq:ieq} to see that $\phi_\infty$ is an
  $L^{2,p}$ solution of 
\[ \det(\partial_j\overline{\partial}_k\phi_\infty) =
e^{-\lambda\phi_\infty} |U_\infty|^2 \]
 on $0.7B^{2n}$ for some $\lambda \leqslant 1$ and nowhere-vanishing
 holomorphic function $U_\infty$. Using B\l{}ocki~\cite[Theorem
 2.5]{Bl99} this implies that $\phi_\infty$ is $C^{2,\alpha}$, and it
 follows that $\omega_\infty$ is a smooth K\"ahler-Einstein metric. In
 addition, passing to the limit in \eqref{eq:upbound} and using that $t_i\to
  1$, we get that the metric $\omega_\infty$ satisfies
  \[ \label{eq:upbound2} K^{-1} \omega_{Euc} \leqslant\omega_\infty \leqslant
  C'\omega_{Euc}. \]
  It remains to show that the $\omega_i$ converge to $\omega_\infty$ in
  the Gromov-Hausdorff sense on a smaller ball, 
  since using Anderson's result, we will
  then obtain a $(\sigma,\theta)$-chart centered at $p$ for
  sufficiently large $i$, contradicting our assumption that no such
  chart exists. 
 
For the Gromov-Hausdorff convergence, let us denote by $
d_i, d_\infty$ the distance functions induced by $
\omega_i, \omega_\infty$. Given $\epsilon > 0$ we will show that
\[\label{eq:dq1q2} d_i(q_1, q_2) \leqslant d_\infty(q_1,q_2) + \epsilon\] 
for any
$q_1,q_2\in \frac{1}{4}B^{2n}$ and
sufficiently large $i$. The converse inequality will be
analogous. 

For $\delta > 0$, let us denote by
$V_\delta$ the Euclidean $\delta$-neighborhood of $V$. Note that by
\eqref{eq:upbound} we have
\[ \mathrm{Vol}(V_\delta, \omega_i) = \int_{V_\delta} \omega_i^n
\leqslant (C')^n \int_{V_\delta} e^{-(1-t_i)nf_i}\,\omega_{Euc}^n, \]
and so for large enough $i$, by H\"older's inequality we obtain
\[ \mathrm{Vol}(V_\delta, \omega_i) \leqslant \Psi(\delta), \]
where $\Psi(\delta)$ denotes a function which converges to zero with
$\delta$, and which we might change below. 
By non-collapsing,
$\mathrm{Vol}(B_r(q_j, \omega_i)) > cr^{2n}$ for $j=1,2$, and so
there is a function
$\Psi(\delta)$, converging to zero with $\delta$, such that if $q_j\in
V_\delta$, then the ball
$B_{\Psi(\delta)}(q_j, \omega_i)$ intersects the boundary of $V_\delta$. In other
words, there are points $q_j' \not\in V_\delta$ such that
$d_i(q_j, q_j') < \Psi(\delta)$ for $j=1,2$ and sufficiently large
$i$. This means that we can replace $q_j$ by $q_j' \in B\setminus
V_\delta$ changing the distance $d_i(q_1, q_2)$ by only
$\Psi(\delta)$. That $d_\infty(q_1,q_2)$ also only changes by
$\Psi(\delta)$ follows from \eqref{eq:upbound} and
\eqref{eq:upbound2}. 

This means that we can assume that $q_1,q_2 \not\in V_\delta$. We will use
Cheeger-Colding's segment inequality, \cite[Theorem 2.15]{Ch01}, to
show that for sufficiently large $i$ 
we can find $q_j'$ for $j=1,2$ such that $d_{Euc}(q_j, q_j') <
\delta /2$, satisfying
\[ d_i(q_1', q_2') \leqslant d_\infty(q_1', q_2') +
\frac{\epsilon}{2}. \]
Note that the convergence of $\omega_i$ to $\omega_\infty$ is $C^{1,\alpha}$
outside $V_{\delta / 2}$ and $\omega_\infty$ is uniformly equivalent
to the Euclidean metric, so
\[ d_i(q_j, q_j'), \,\,d_\infty(q_j, q_j') \leqslant \Psi(\delta). \]
It then follows from this that
\[ d_i(q_1, q_2) \leqslant d_\infty(q_1, q_2) + \frac{\epsilon}{2} +
\Psi(\delta), \]
which implies the result we want once $\delta$ is sufficiently small. 

Let $g : 0.7B^{2n} \to \mathbf{R}$ be
the function 
\[ g(x) = \sup_{\substack{v\in T_xB \\ |v|_{\omega_\infty}=1}} \Big|
|v|_{\omega_i} - |v|_{\omega_\infty}\Big|. \]
The $L^p$-convergence implies that by choosing $i$ sufficiently large,
we can make $\int_B g\, \omega_\infty^n$ arbitrarily small. Let
$\gamma:[0,l] \to B^{2n}$ be a unit speed minimizing geodesic
from $q_1$ to $q_2$ with respect to $\omega_\infty$. We then have
\[ d_i(q_1, q_2) \leqslant d_\infty(q_1, q_2) + \int_0^l
g(\gamma(\tau))\,d\tau. \]

The segment inequality allows us to perturb $q_1,q_2$ slightly such
that  the corresponding integral above is very small. Indeed, in the notation
of \cite[Theorem 2.1]{Ch01} we let $A_i = B_{\delta / 2}(q_i)$ for
$i=1,2$ be Euclidean balls of radius $\delta/2$. Then
\[ \int_{A_1\times A_2} \mathcal{F}_g(x_1, x_2) \leqslant
C\Big(\mathrm{Vol}(A_1) + \mathrm{Vol}(A_2)\Big)\,\int_{0.7B} g, \]
where the integrals and volumes are with respect to $\omega_\infty$
(which is uniformly equivalent to the Euclidean metric) and
\[ \mathcal{F}_g(x_1, x_2) = \inf_\gamma \int_0^l g(\gamma(s))\,ds \]
is an infimum over all minimizing geodesics $\gamma$ from $x_1$ to
$x_2$ and $l$ is the length of $\gamma$. 

Given $\delta$, we control the volumes of $A_1$ and $A_2$ (here these
volumes only depend on $\delta$
since we are using $\omega_\infty$, but for the converse of
\eqref{eq:dq1q2} we need the volumes using the metric $\omega_i$,
which are nevertheless controlled by the $L^p$-convergence of
$\omega_i$ to $\omega_\infty$). We can then choose $i$ sufficiently
large, so that the average of $\mathcal{F}_g$ satisfies
\[ \frac{1}{\mathrm{Vol}(A_1\times A_2)} \int_{A_1\times A_2}
\mathcal{F}_g(x_1,x_2) \leqslant \frac{\epsilon}{2}, \]
which in turn means that we can find $q_i' \in B_{\delta / 2} (q_i)$
such that $\mathcal{F}_g(q_1', q_2') \leqslant \epsilon / 2$. In
particular 
\[ d_i(q_1', q_2') \leqslant d_\infty(q_1', q_2') +
\frac{\epsilon}{2}, \]
which is what we wanted to show.  
\end{proof}

Finally we need the analog of \cite[Proposition 8]{CDS13_3}. 
\begin{prop}\label{prop:CDSprop8}
 There is a $c > 0$ such that for any $\theta, A > 0$ we can find
 $\delta(\theta), \gamma(A,\theta)$ with the following property. If
\begin{itemize}
\item $1 - I(B) \leqslant \delta$,
\item $\displaystyle{ \int_B \alpha\wedge \omega^{n-1}} \leqslant A$,
\item $t > 1-\gamma$,
\end{itemize}
then there is a $(c,\theta)$-chart centered at $p$. 
\end{prop}
\begin{proof}
  The proof from \cite{CDS13_3} can be used almost verbatim. If for
  some $A > 0$ we have $V(p,1) \leqslant 2A$, 
then denote
\[ \label{eq:Zrdef}
  Z_r = \{ q\in B\,:\, V(q,r) \geqslant A\}. \]
If $\{rB_1,\ldots, rB_k\}$ is a maximal collection of disjoint $r$-balls with centers
in $Z_r$, then $k\leqslant 2r^{2n-2}$. The balls $2rB_i$ cover $Z_r$,
while we can cover each $2rB_i$ with a fixed number of $r$-balls. It
follows that $Z_r$ is covered by at most $Mr^{2n-2}$ balls of
radius $r$, where $M$ is independent of $r, A$. Use this $M$ in
Proposition~\ref{prop:CDS} to obtain $\rho=\rho(M)$. From
Proposition~\ref{prop:CDSprop6} we have a number $\sigma=\sigma(K)$
(with $K$ from Proposition~\ref{prop:CDS}). Define $c=\rho\sigma$, and
use $M,c$ in Proposition~\ref{prop:CDS} to obtain $\eta(M,c)$ and
$\delta(M,c)$. 

If $A$ is sufficiently small, then for small enough $\delta(\theta)$ we can
combine Propositions~\ref{prop:epsilonreg2} and \ref{prop:CDSprop5}
with Anderson's result to get a $(c,\theta)$-chart at $p$. 
Setting $\delta(\theta)$ even smaller, we let $\delta(\theta) <
\delta(M,c)$ and $\delta(\theta) < \delta_*(\theta)$ (from
Proposition~\ref{prop:CDSprop6}).  Let
$\gamma(A,\theta)$ be the constant from
Proposition~\ref{prop:CDSprop6}. We say that $A$ is good, if with
these choices of constants the conclusion of
Proposition~\ref{prop:CDSprop8} holds. If $A$ is very small, then
we have seen that $A$ is good. We can also restrict ourselves to
$\theta < \delta(M,c)$. 

We suppose now that $A$ is good, and show that $2A$ is also good. 
Let $V(p,1) \leqslant 2A$, and 
$W=\bigcap\limits_{\eta \leqslant r < 1}Z_r$ with $Z_r$ as in
\eqref{eq:Zrdef},
so $m(\eta,W) < M$. If $x\in
B\setminus W$, then $V(x,r) < A$ for some $r\in [\eta, 1)$.
By assumption we can apply
Proposition~\ref{prop:CDSprop8} to $B(x, r)$ scaled to unit size
to obtain a $(c\eta, \delta(M,c))$-chart at $x$. This means that
Proposition~\ref{prop:CDS} applies, and its conclusion can be used in
Proposition~\ref{prop:CDSprop6}. This provides a $(\sigma,
\theta)$-chart at $p$, so $2A$ is also good. 
\end{proof}

Given this Proposition, the argument in \cite[Section 2.6]{CDS13_3}
can be used verbatim to show the following.
Let $\omega_t$ be metrics along the continuity method satisfying
\[ \mathrm{Ric}(\omega_t) = t\mathrm{Ric}(\omega_t) + (1-t)\alpha, \]
and let $Z$ be the Gromov-Hausdorff limit of $(M, \omega_{t_i})$ for
some $t_i\to 1$. Then the regular set in $Z$ is open, and the
convergence of the metrics is locally in $L^p$. The same applies to
any iterated tangent cone of $Z$. Indeed, suppose that $p$ is a
regular point in an iterated tangent cone of $Z$. Fix a small $\theta$
and let $\delta=\delta(\theta)$ from Proposition~\ref{prop:CDSprop8}.
A suitable ball $B$ centered at $p$, scaled to unit size then satisfies $1-I(B) <
\delta/3$. By definition this ball is the Gromov-Hausdorff limit of a
sequence of balls $B_i \subset Z$, scaled to unit size. In particular,
$1-I(B_i) < 2\delta/3$ for sufficiently large $i$. For a fixed such
$i$, the ball $B_i$ is the Gromov-Hausdorff limit of balls $B_{i,j}
\subset (M, \omega_{t_j})$, with some radius $r_i$ (the radius being
independent of $j$). For sufficiently large $j$ (depending on $i$), we
will have $1-I(B_{i,j}) < \delta$, and at the same time 
\[ r_i^{2-2n}\int_{B_{i,j}} \alpha\wedge \omega_{t_i}^{n-1} \leqslant
r_i^{2-2n} \int_M \alpha\wedge \omega_{t_i}^{n-1} < C'r_i^{2-2n}, \]
for some $C'$. For fixed $i$ we can therefore apply
Proposition~\ref{prop:CDSprop8} to the balls $B_{i,j}$ scaled to
unit size, for sufficiently large $j$ to obtain a $(c,\theta)$-chart
at their centers. Since the original ball $B$ in the iterated tangent
cone is a limit of a sequence $B_{i, j(i)}$ where $j(i)$ can be taken
arbitrarily large for any $i$, we obtain the required convergence on
the ball $cB$. 

\begin{prop}\label{prop:iterated}
  In the limit space $Z$ above, no iterated tangent cone can be of the form
  $\mathbf{C}_\gamma\times \mathbf{C}^{n-1}$. 
\end{prop}
\begin{proof}
  First, the discussion above means that
  one can use the arguments of \cite[Section
2.5]{CDS13_2} to ensure that we are in the setting discussed before
Proposition~\ref{prop:conebound}. In other words, if an iterated
tangent cone of 
$Z$ is $\mathbf{C}_\gamma\times \mathbf{C}^{n-1}$, then we can first
find a ball in $Z$ which scaled to unit size is very close to the unit
ball in $\mathbf{C}_\gamma\times \mathbf{C}^{n-1}$, and this ball is a
Gromov-Hausdorff limit of balls of some fixed radius $r$ in $(M,
\omega_{t_i})$.  After scaling up by a fixed
factor the metrics $\widetilde{\omega}_i = k\omega_i$ on these small
balls  can then be thought
of as metrics on the Euclidean ball $B^{2n}$, and they satisfy the
properties (1),(2),(3) before
Proposition~\ref{prop:conebound}. Consider the set
\[ V = \{(u,v_1,\ldots, v_{n-1})\,:\, |u|\leqslant 1/4, |v_1|^2 +
\ldots + |v_{n-1}|^2 \leqslant 1/4\}. \]
We want to bound
\[ \int_V \mathrm{Ric}(\widetilde{\omega_i})\wedge
\omega_{Euc}^{n-1} \]
from below, and for this it is enough, for each $|\mathbf{a}|^2 \leqslant 1/4$,
to bound
\[ \label{eq:intVa} \int_{V_\mathbf{a}}\mathrm{Ric}(\widetilde{\omega_i}) \]
from below where $V_{\mathbf{a}}$
is the disk $V\cap \{(v_1,\ldots,v_{n-1})=\mathbf{a}\}$. Using that
\[ \mathrm{Ric}(\widetilde{\omega_i}) =
-\ddb\log\frac{\widetilde{\omega}_i^n}{ \omega_{Euc}^n}, \]
the integral \eqref{eq:intVa} can be computed by an integral over
$\partial V_{\mathbf{a}}$ of a term involving one derivative of
$\widetilde{\omega}_i$. Since $\partial V_{\mathbf{a}}$ is a compact set disjoint
from $\{u=0\}$, we can assume that $\widetilde{\omega}_i$ is very
close to the cone metric $\eta_\gamma$ in $C^{1,\alpha}$, so the
integral \eqref{eq:intVa} can be assumed to be very close to the
corresponding integral for $\eta_\gamma$. Unless $\gamma=1$, this
latter integral is non-zero. So for a sufficiently large scaling
factor $k$, and large enough $i$, we have a
lower bound
\[ \int_V \mathrm{Ric}(\widetilde{\omega_i})\wedge
\omega_{Euc}^{n-1} > c_0, \]
which because of $\omega_{Euc} < C\widetilde{\omega}_i$ implies
\[ \label{eq:c1}
\int_V \mathrm{Ric}(\widetilde{\omega_i})\wedge
\widetilde{\omega}_i^{n-1} > c_1, \]
for some other constant $c_1$ (depending on $\gamma$). We have 
\[\mathrm{Ric}(\widetilde{\omega}_i) = k^{-1}t_i\widetilde{\omega}_i +
(1-t_i)\alpha_i \leqslant
k^{-1}\widetilde{\omega}_i + (1-t_i)\alpha_i, \]
where we write $\alpha_i$ to emphasize that the (fixed) K\"ahler
metric $\alpha$ will depend on $i$ when we are thinking of the
$\widetilde{\omega}_i$ as metrics on the Euclidean ball $B^{2n}$. By
volume convergence, we control the $\widetilde{\omega}_i$-volume of
the set $V$, so from \eqref{eq:c1} we get
\[\label{eq:c1/2} \int_V (1-t_i)\alpha_i \wedge \widetilde{\omega}_i^{n-1} >
\frac{c_1}{2} \]
once the scaling factor $k$ is large enough (and also $i$ is
sufficiently large). But
\[ \int_V \alpha_i\wedge \widetilde{\omega}_i^{n-1} \leqslant
k^{n-1}\int_M \alpha\wedge\omega_{t_i}^{n-1}, \]
which contradicts \eqref{eq:c1/2} as $t_i\to 1$. 
\end{proof}

From this result it is now clear that every tangent cone of $Z$ is
good, since the singular set in each tangent cone must have Hausdorff
codimension at least 4 and so the argument of \cite[Proposition
3.5]{DS12} applies.  This completes the proof of
Theorem~\ref{thm:goodcones}. 

\section{Proof of Corollary~\ref{cor:Paul}}\label{sec:corproof}
We now give the proof of Corollary~\ref{cor:Paul}, relating Paul's
version of stability with the existence of K\"ahler-Einstein metrics. 
\begin{proof}
Paul~\cite{Paul13} shows that under the assumption of his version of
stability, for each $l > 0$, the
Mabuchi energy is proper when restricted to the space of Bergman
metrics, i.e. metrics obtained as pullbacks of the Fubini-Study metric
under an embedding using $K^{-l}_M$. As explained in \cite{Tian12_1},
the partial $C^0$-estimate allows us to compare the Mabuchi energy
of the metrics $\omega_t$ with the Bergman metrics obtained using
embeddings with $L^2$-orthonormal sections. The following argument is
similar to that in 
Tian-Zhang~\cite{TZ13}. Recall that if $\eta' = \eta + \ddb\phi$ are
two metrics in $c_1(M)$,
then the Mabuchi energy is defined by 
\[ \mathcal{M}(\eta, \eta') = \int_0^1 \int_{M} \phi (n -
S(\eta_t))\,\eta_t^n\,dt, \]
where $\eta_t = \eta + t\ddb\phi$. 
Let us write $t_i \to T$ for a sequence converging to $T$, and let
$\omega_i = \omega_{t_i}$. For the $k_0$ in Theorem~\ref{thm:main}
write
\[ \eta_i = \omega_i + \frac{1}{k_0}\ddb \rho_{\omega_i, k_0} \]
for the corresponding Bergman metric, i.e. $\eta_i = \frac{1}{k_0}
\Phi_i^*\omega_{FS}$, where $\Phi_i : M \to\mathbf{CP}^N$ are
embeddings given by an orthonormal basis of $H^0(K_M^{-k_0})$ with
respect to the metric $\omega_i$. We can replace our sequence $t_i$ by
a subsequence to ensure that the varieties $\Phi_i(M)$ converge to a
limit $Z$ in projective space. Then Paul's result cited above
implies that if $M$ is stable, and $Z$ is not in the
$GL(N+1,\mathbf{C})$-orbit of the $\Phi_i(M)$, then
$\mathcal{M}(\eta_1, \eta_i)\to\infty$. 

The cocycle property of the Mabuchi energy implies that
\[ \mathcal{M}(\eta_1, \eta_i) = \mathcal{M}(\eta_1, \omega_i) +
\mathcal{M}(\omega_i, \eta_i), \]
and along the continuity method the Mabuchi energy is decreasing, so
if $\mathcal{M}(\eta_1, \eta_i) \to \infty$, then necessarily
$\mathcal{M}(\omega_i, \eta_i)\to \infty$. We will thus obtain a
contradiction if we show that $\mathcal{M}(\omega_i, \eta_i)$ is
bounded. To see this, as in \cite{TZ13} we can use the explicit
formula for the Mabuchi energy from Tian~\cite{Tian00} to obtain
\[ \mathcal{M}(\omega_i, \eta_i) \leqslant \int_M \log
\frac{\eta_i^n}{\omega_i^n} \eta_i^n + \int_M u_i (\eta_i^n -
\omega_i^n), \]
where $u_i$ is the Ricci potential of $\omega_i$ defined by
\[ \ddb u_i = \omega_i - \mathrm{Ric}(\omega_i). \]
We can normalize $u_i$ so that $\inf_M u_i = 0$.  From the defining
equation we have $\Delta u_i = n - S(\omega_i) \leqslant n$, which
implies, through control of the Green's function,  that
\[ \int_M u_i \omega_i^n \leqslant C. \]
In addition, the partial $C^0$-estimate combined with a gradient
estimate for holomorphic sections implies that $\eta_i <
C\omega_i$ (see Donaldson-Sun~\cite[Lemma 4.2]{DS12}). From this we
easily obtain an upper bound on $\mathcal{M}(\omega_i, \eta_i)$. 
It follows therefore that the $\Phi_i(M)$ converge to a limit in the
$GL(N+1,\mathbf{C})$-orbit of $\Phi_1(M)$, which in turn can be used
to control the K\"ahler potentials of the $\eta_i$. The partial
$C^0$-estimate then in turn controls the K\"ahler potentials of the
$\omega_i$. It follows that the continuity method can be continued
past $T$ (or we obtain a K\"ahler-Einstein metric when $T=1$) 
using also the estimates of
Yau~\cite{Yau78} and Aubin~\cite{Aub78}. 

Conversely an important result due to Tian~\cite{Tian97} is that if $M$ admits a
K\"ahler-Einstein metric and has no holomorphic vector fields, then
the Mabuchi energy is proper on the space of all K\"ahler metrics in
$c_1(M)$, and this (see also Phong-Song-Sturm-Weinkove~\cite{PSSW} for an
improvement), implies the stability of $M$ in the sense of
Paul. See also the discussion in Section 6 of Tian~\cite{Tian13} and
\cite{Tian13_1}. 
\end{proof}

\providecommand{\bysame}{\leavevmode\hbox to3em{\hrulefill}\thinspace}
\providecommand{\MR}{\relax\ifhmode\unskip\space\fi MR }
\providecommand{\MRhref}[2]{%
  \href{http://www.ams.org/mathscinet-getitem?mr=#1}{#2}
}
\providecommand{\href}[2]{#2}

\end{document}